\documentclass{article}
\usepackage{amsthm,amsmath}
\usepackage{amssymb}
\usepackage{url}
\usepackage{color}
\newtheorem{theorem}{Theorem}[section]
\newtheorem{corollary}[theorem]{Corollary}
\newtheorem{lemma}[theorem]{Lemma}

\newtheorem{definition}[theorem]{Definition}
\newtheorem{hypothesis}[theorem]{Hypothesis}

\newtheorem{claim}{Claim}
\newtheorem{question}[theorem]{Question}
\newcommand{\Soc}{\mathrm{Soc}}
\newcommand{\G}{\mathrm{G}}
\newcommand{\Sz}{\mathrm{Sz}}
\newcommand{\Sym}{\mathrm{Sym}}
\newcommand{\Aut}{\mathrm{Aut}}
\newcommand{\Sp}{\mathrm{Sp}}
\newcommand{\PSL}{\mathrm{PSL}}
\newcommand{\A}{\mathrm{A}}
\newcommand{\C}{\mathrm{C}}
\newcommand{\D}{\mathrm{D}}
\newcommand{\Cos}{\mathrm{Cos}}
\newcommand{\Stab}{\mathrm{Stab}}
\newcommand{\Fix}{\mathrm{Fix}}
\newcommand{\N}{\mathrm{N}}
\newcommand{\Ree}{\mathrm{Ree}}

\begin{document}
\title{Vertex-primitive \(s\)-arc-transitive digraphs admitting a Suzuki or Ree group}
\date{}
\author{Lei Chen \quad \quad Michael Giudici \quad\quad Cheryl E. Praeger\\ \\
Department of Mathematics and Statistics\\
The University of Western Australia\\
35 Stirling Highway, Perth WA 6009\\
Australia\\ \\
\small Lei.Chen@research.uwa.edu.au\thanks{We thank the two anonymous referees for their careful reading of the paper and their advice}\thanks{\(^*\) The corresponding author}\\ 
\small Michael.Giudici@uwa.edu.au\\
\small Cheryl.Praeger@uwa.edu.au}

\maketitle

\begin{abstract}
The investigation of \(s\)-arc-transitivity of digraphs can be dated back to 1989 when the third author showed that \(s\) can be arbitrarily large if the action on vertices is imprimitive. However, the situation is completely different when the digraph is vertex-primitive and not a directed cycle. In 2017 the second author, Li and Xia constructed the first infinite family of \(G\)-vertex-primitive \(2\)-arc-transitive examples, and asked if there is an upper bound on $s$ for  \(G\)-vertex-primitive $s$-arc-transitive digraphs that are not directed cycles. In 2018 the second author and Xia showed that if there is a largest such value of $s$ then it will occur when \(G\) is almost simple. So far it has been shown that $s\leqslant 2$ for almost simple groups whose socle is an alternating group or a projective special linear group. 
The contribution of this paper is to prove that \(s\leq1\) in the case of the Suzuki groups and the small Ree groups. We give constructions with \(s=1\) to show that the bound is sharp.

\end{abstract}

\section{Introduction}

The property of \(s\)-arc-transitivity has been well-studied for many years. Weiss \cite{W} proved that finite undirected graphs that are not cycles can be at most 7-arc-transitive. On the other hand, the third author \cite{Praeger} showed that for each \(s\) there are infinitely many finite \(s\)-arc-transitive digraphs that are not \((s+1)\)-arc-transitive. 

However, vertex-primitive \(s\)-arc-transitive digraphs for large \(s\) seem rare. Though extensive attempts had been made to find a vertex-primitive \(s\)-arc-transitive digraph for \(s\geq2\), no such examples were found until 2017 when the second author, with Li and Xia, constructed an infinite family of $2$-arc transitive examples in  \cite{example}. They also asked:
\begin{question}\label{qu}
Is there an upper bound on \(s\) for vertex-primitive \(s\)-arc-transitive digraphs that are not directed cycles?
\end{question}

A group \(G\) is said to be an \emph{almost simple group} if it has a unique \emph{minimal normal subgroup} \(T\) such that \(T\) is a nonabelian simple group. This implies (identifying $T$ with the group of inner automorphisms of $T$) that \(T\triangleleft G\leqslant \Aut(T)\).

We believe that the answer to Question \ref{qu} is yes.   Proving this and determining the value of the upper bound has been reduced in \cite[Corollary 1.6]{quasi} to the case where the vertex-primitive automorphism group is  almost simple. This has motivated the study of Question \ref{qu} for various families of almost simple groups. It has been shown that $s\leqslant 2$ if \(G=\) S\(_{m}\) or \(\A_{m}\) \cite{chen,alter}, or if \(\PSL_{n}(q)\leqslant G\leqslant \Aut(\PSL_{n}(q))\) \cite{linear}.

This paper determines an upper bound \(s\) for vertex-primitive \(s\)-arc-transitive digraphs whose automorphism groups are almost simple Ree or Suzuki groups, thus settling Question \ref{qu} for these two families of almost simple groups.

\begin{theorem}\label{general}
Let \(s\) be a non-negative integer and let \(\Gamma\) be a \(G\)-vertex-primitive \((G,s)\)-arc-transitive digraph, where \(G\) is almost simple with socle \(\Ree(3^{2n+1})\) or \(\Sz(2^{2n+1})\). Then \(s\leq 1\).
\end{theorem}
 
  We juxtapose the Suzuki groups and the Ree groups in this paper as many similarities can be found between these two kinds of exceptional simple groups: (1) the Suzuki groups \(\Sz(2^{2n+1})\) bear a relation to the symplectic groups \(\Sp_{4}(2^{2n+1})\) similar  to that of the  Ree groups \(\Ree(3^{2n+1})\) to \(\G_{2}(3^{2n+1})\); (2) the maximal subgroup types of the Suzuki groups and the  Ree groups are fairly similar; (3) the only outer automorphisms of the two groups are field automorphisms. Hence we are able to apply similar arguments to both.

We now remind readers of some terms mentioned above.
A \emph{digraph} \(\Gamma\) is a pair \((V,\to)\) such that \(V\) is the set of vertices and \(\to\) is an anti-symmetric and irreflexive relation on \(V\). For a non-negative integer \(s\), we call a sequence \(v_{0}, v_{1},\dots, v_{s}\) in \(V\) an \emph{\(s\)-arc} if \(v_{i}\to v_{i+1}\) for each \(i\in\{0,1,\dots,s-1\}\). Note that a 1-arc is simply called an \emph{arc}. For \(G\leqslant \Aut(\Gamma)\), we say that \(\Gamma\) is a \((G,s)\)-arc-transitive digraph if \(G\) acts transitively on the set of \(s\)-arcs of \(\Gamma\). We note that an \((s+1)\)-arc-transitive digraph is naturally \(s\)-arc-transitive if every \(s\)-arc extends to an \((s+1)\)-arc. Note that \((s+1)\)-arc-transitive implies \(s\)-arc-transitive for a \(G\)-vertex-primitive digraph. A transitive subgroup \(G\leqslant \Sym(\Omega)\) is said to be primitive if it does not preserve any non-trivial partition of \(\Omega\). For \(G\leqslant \Aut(\Gamma)\), we say that \(\Gamma\) is \emph{\(G\)-vertex-primitive} if \(G\) acts primitively on \(V\). A digraph is said to be \emph{finite} if \(|V|\) is finite and all the digraphs we consider in this paper will be finite. Note that a graph \(\Sigma\) is a pair \((V,\sim)\) such that \(V\) is the vertex set and \(\sim\) is a symmetric and irreflexive relation on \(V\). As in the case of digraphs, a sequence \(v_{0}, v_{1},\ldots, v_{s}\) is an \emph{\(s\)-arc} if \(v_{i}\sim v_{i+1}\) for \(0\leq i\leq s-1\) and \(v_{i}\neq v_{i+2}\) for each \(i\). A graph \(\Sigma\) is said to \emph{\(s\)-arc-transitive} if the automorphism group of \(\Sigma\) acts transitively on the set of \(s\)-arcs.

\section{Preliminaries}
\subsection{Notation}
We begin by defining some group theoretic notation:
\\

For a group \(X\), we denote by \(\Soc(X)\) the socle of \(X\), and by \(\Pi(X)\) the set of prime divisors of \(|X|\). For a prime number \(p\) and an integer \(n\), we denote by \(n_{p}\) the \(p\)-part of \(n\), which is the largest power of \(p\) dividing \(n\).
\\

The expression \(n\) or \(\C_{n}\) denotes a cyclic group of order \(n\) while \([n]\) denotes an unspecified group of order \(n\). The expression \(p^{n}\) denotes an elementary abelian group of order \(p^{n}\), that is, a direct product of \(n\) copies of \(\C_{p}\).

Extensions of groups are written in one of the following ways: \(A\times B\) denotes a direct product of \(A\) and \(B\); also \(A:B\) denotes a semidirect product of \(A\) by \(B\); and \(A.B\) denotes an unspecified extension of \(A\) by \(B\).

For groups \(A\) and \(B\) such that \(B\leqslant A\), we denote by \(\N_{A}(B)\) the normaliser of \(B\) in \(A\), and \(\C_{A}(B)\) the centraliser of \(B\) in \(A\).

\begin{lemma}{\rm \cite[Lemma 2.1]{linear}\label{Zsigmondy}}
For any positive integer \(n\) and prime \(p\), we have \((n!)_{p}<p^{\frac{n}{p-1}}\).

\end{lemma}
\begin{definition}
{\rm Given integers \(a, m\geq 2\), a prime \(r\) is said to be a \emph{primitive prime divisor} of \(a^{m}-1\) if \(r\) divides \(a^{m}-1\) and does not divide \(a^{i}-1\) for any \(i<m\).}
\end{definition}
For \(r\) a primitive prime divisor of \(a^{m}-1\), we conclude by Fermat's Little Theorem that \(r\equiv1\pmod{m}\), and therefore \(r>m\).

\begin{lemma}\label{ppd}{\rm \cite[Theorem IX.8.3]{bb}}
For \(a, m\geq2\), there exists a primitive prime divisor of \(a^{m}-1\) except when \((a,m)=(2,6)\), or \(a+1\) is a power of $2$ and \(m=2\).
\end{lemma}

\subsection{Group factorisations}
A factorisation of a group \(G\) is an expression of \(G\) as the product of two subgroups \(A\) and \(B\) of \(G\), where \(A\) and \(B\) are called factors. A proper group factorisation occurs when neither \(A\) nor \(B\) equals \(G\). 

\begin{definition}
{\rm A factorisation \(G=AB\) is called a \emph{homogeneous factorisation} of \(G\) if it is proper and  \(A\) is isomorphic to \(B\).}
\end{definition}

We now give two technical lemmas, which will be useful later.

\begin{lemma}\label{proj}
Suppose that \(G=\langle H, x\rangle\) with \(H\triangleleft G\). Suppose that \(K\leqslant G\) and \(K=AB\) is a homogeneous factorisation such that \(B=A^{t}\) for some \(t\in G\), and let \(\pi:G\to G/H\) denote the natural projection map. Then \(\pi(A)=\pi(B)=\pi(K)\).
\end{lemma}

\begin{proof}
Let \(m\) be the smallest positive integer such that \(x^{m}\in H\). Note that \(G/H=\langle Hx\rangle\) has order \(m\) by the minimality of \(m\). Since \(\pi(A)= AH/H\leqslant G/H\) and \(G/H\) is cyclic, we conclude that \(\pi(A)=\langle Hx^{j}\rangle\) for some divisor \(j\) of \(m\). So there exists \(a\in A\) such that \(\pi(a)=Hx^{j}\). Thus \(Ha=Hx^{j}\), so \(a=hx^{j}\) for some \(h\in H\).

As \(B=A^{t}\) with \(t\in G\), we have \(a^{t}\in B\) and \(\pi(a^{t})=\pi(t^{-1})\pi(a)\pi(t)=\pi(a)=Hx^{j}\) since \(G/H\) is abelian. Hence \(\pi(B)\geqslant\pi(A)\). The same argument with \(A\) and \(B\) interchanged and \(t\) replaced by \(t^{-1}\), gives that \(\pi(A)\geqslant\pi(B)\). Hence \(\pi(A)=\pi(B)\), and so \(\pi(K)=\pi(A)\pi(B)=\pi(A)=\pi(B)\).
\end{proof}

\begin{lemma}\label{ps}
Suppose that \(G=AB\) with \(G=\PSL_{2}(8):3\) such that \(A, B\) are proper subgroups of \(G\). Then \(|A|\neq |B|\).
\end{lemma}
\begin{proof}
Suppose for a contradiction that there exists a factorisation \(G=AB\) with \(G=\PSL_{2}(8):3\) and \(|A|=|B|\). Then we deduce that \(|A|_{p}^{2}=|B|_{p}^{2}\geq |G|_{p}\) for any prime \(p\). In particular, \(|A|_{2}\geq 2^{2}\), \(|A|_{3}\geq 3^{2}\) and \(|A|_{7}=7\). On the other hand, since \(A=A\cap \PSL_{2}(8)\) or \(A=(A\cap \PSL_{2}(8)).3\). We therefore conclude that \(|A\cap \PSL_{2}(8)|\) is divisible by 2, 3 and 7. By [1, Corollary 5 and Table 10.7] there are no proper subgroups of \(\PSL_{2}(8)\) with order divisible by 2, 3 and 7, and hence \(A\cap \PSL_{2}(8)=\PSL_{2}(8)\). Similarly, we conclude that \(B\cap \PSL_{2}(8)=\PSL_{2}(8)\). However, since \(|A|=|B|\), we must have that \(A=B=G\), which contradicts the fact that \(A\) and \(B\) are proper subgroups of \(G\).
\end{proof}

\subsection{Arc-transitivity}
We say that a group \(G\) acts on a digraph \(\Gamma\) if \(G\leq \Aut(\Gamma)\). Here are two results in \cite{linear} and \cite{quasi} that reveal some important properties for an \(s\)-arc-transitive digraph \(\Gamma\) where \(s\geq2\).

\begin{lemma}{\rm \cite[Lemma 2.2]{quasi}\label{0}}
Let \(\Gamma\) be a digraph, and \(v_{0}\rightarrow v_{1}\rightarrow v_{2}\) be a \(2\)-arc of \(\Gamma\). Suppose that \(G\) acts arc-transitively on \(\Gamma\). Then \(G\) acts \(2\)-arc-transitively on \(\Gamma\) if and only if \(G_{v_{1}}=G_{v_{0}v_{1}}G_{v_{1}v_{2}}\). Moreover, there exists some \(t\in G\) such that \((G_{v_{0}v_{1}})^{t}=G_{v_{1}v_{2}}\). 

\end{lemma}
\begin{lemma}{\rm \cite[Lemma 2.14]{linear}\label{factor}}
Let \(\Gamma\) be a connected \(G\)-arc-transitive digraph with arc \(v\rightarrow w\). Let \(g\in G\) such that \(v^{g}=w\). Then \(g\) normalises no proper nontrivial normal subgroup of \(G_{v}\).
\end{lemma}

We now set out the following hypothesis that we will use throughout the paper.

\begin{hypothesis}\label{ga}
Let \(\Gamma\) be a vertex-primitive \((G,s)\)-arc-transitive digraph for some \(s\geq 2\), and let \(u\to v\to w\) be a \(2\)-arc and \(g\in G\) such that \((u,v)^{g}=(v,w)\). Then by Lemma \(\ref{0}\), \((G_{uv})^{g}=G_{vw}\) and \(G_{v}=G_{uv}G_{vw}\) is a homogeneous factorisation.

\end{hypothesis}

Note that necessary conditions for a digraph \(\Gamma\) to be \(G\)-vertex-primitive and \((G,2)\)-arc-transitive are that \(G_{v}\) is a maximal core-free subgroup of $G$ and that \(G_{v}\) admits a homogeneous factorisation. Therefore, to disprove the 2-arc-transitivity of \(G\) it suffices for us to show that a maximal core-free subgroup \(G_{v}\) does not have a homogeneous factorisation.

We have the following corollary to Lemma \ref{factor}.

\begin{corollary}\label{Q}
Suppose that Hypothesis \(\ref{ga}\) holds. Then, for each prime \(p\) dividing \(|G_{v}|\), \(G_{v}\) has at least two subgroups of order \(p\). 
\end{corollary}

\begin{proof}
Suppose for a contradiction that there exists a prime \(p\) such that \(G_{v}\) has a unique subgroup \(Q\) of order \(p\). Then \(Q\) is normal in \(G_{v}\). By Hypothesis \ref{ga}, \(G_{v}=AB\) where \(A=G_{uv}\) and \(B=G_{vw}\), so $|G_v|$ divides $|A|\cdot |B|=|A|^2$ and hence \(|A|_{p}=|B|_{p}\geq p\). Hence both of \(A\) and \(B\) contain the unique subgroup \(Q\) of order \(p\), and as \(A^{g}=B\), we have \(Q^{g}\leqslant B\leqslant G_{v}\) and therefore \(Q^{g}=Q\), contradicting Lemma \(\ref{factor}\).
\end{proof}

\section{The small Ree groups}\label{s:Ree}

Suppose that $\Gamma$ is a $G$-vertex-primitive, $(G,s)$-arc-transitive digraph such that $\Soc(G)=\Ree(q)$ with \(q=3^{2n+1}\), for some \(n\geq 1\) and $s\geq 1$. Note that for \(\Ree(q)\leqslant G\leqslant \Aut(\Ree(q))\), since \(\Aut(\Ree(q))=\Ree(q):(2n+1)\) we have \(G=\Ree(q):m\) for some divisor \(m\) of \(2n+1\).
Since the action of \(G\) on \(\Gamma\) is vertex-primitive, the vertex stabiliser \(G_{v}\) is maximal in $G$ and  does not contain $\Ree(q)$. The following list of the maximal subgroups of an almost simple group of socle \(\Ree(q)\) may be found in \cite[Table 8.43]{holt}.



\begin{theorem}\label{reemax}
For \(G=\Ree(q):m\), where \(q=3^{2n+1}\) and \(m\) divides \(2n+1\), the maximal subgroups of \(G\) not containing \(\Ree(q)\) are (up to conjugacy):
\\(i) \(([q^{3}]:\C_{q-1}):m\),
\\(ii) \((2\times \PSL_{2}(q)):m\),
\\(iii) \(((2^{2}\times \D_{\frac{q+1}{2}}):3).m\),
\\(iv) \((\C_{q-\sqrt{3q}+1}:6).m\),
\\(v) \((\C_{q+\sqrt{3q}+1}:6).m\),
\\(vi) \(\Ree(q_{0}):m\), where \(q=q_{0}^{r}\) and \(r\) is prime.
\end{theorem}

For the rest of this section we assume that $s\geq2$, and hence Hypothesis \ref{ga} holds for \(G=\Ree(q):m\), where $q=3^{2n+1}>3$ and $m$ divides $2n+1$, and we let \(L=\Soc(G)=\Ree(q)\). We consider separately each of the possibilities for the maximal subgroup $G_v$ according to Theorem~\ref{reemax},  and in each case derive a contradiction, hence proving that \(s\leq 1\).

 We let \(\pi\) be the natural projection map \(\pi:G_v\rightarrow G_v/L_v\). Note that since $G=LG_v$ we have $\pi(G_v)\cong G/L\cong C_m$.
We note in particular that, by Hypothesis \ref{ga}, $G_v$ has a homogeneous factorisation 
\begin{equation}\label{e:homf}
\mbox{$G_v=AB$ where $A=G_{uv}$ and $B=G_{vw}$ with $A^g=B$ for some $g\in G$.}
\end{equation}
This implies, first that $\Pi(A)=\Pi(B)=\Pi(G_v)$, and secondly, by Corollary~\ref{Q}, that for each prime $p$ dividing $|G_v|$, $G_v$ has at least two subgroups of order $p$.
We use these facts several times in our arguments.

\begin{lemma}\label{r2}
\(G_{v}\) is not a Type (ii) subgroup of \(G\). 
\end{lemma}
\begin{proof}
Suppose to the contrary that \(G_{v}\) is a Type (ii) subgroup of \(G\), and consider the homogeneous factorisation $G_v=AB$ in \eqref{e:homf}, so \(\Pi(A)=\Pi(B)=\Pi(G_{v})\).
Let  \(S\) and \(T\) denote the subgroups of \(L_{v}=L\cap G_{v}\) isomorphic to 2 and \(\PSL_{2}(q)\), respectively. Then \(L_{v}=S\times T\).

Note that, by Lemma \ref{ppd}, there exists 
\(p\in \Pi(G_{v})\) such that \(p\) is a primitive prime divisor of \(3^{2(2n+1)}-1\),  
which is greater than \(2(2n+1)\). Hence \(|A|\) is divisible by \(p\) and, in particular, \(p\) divides the order of \(A_{1}:=A\cap T=A\cap \PSL_{2}(q)\). 
We also notice that
\begin{equation}
\frac{|A||B|}{|A\cap B|}=|AB|=|G_{v}|\nonumber
\end{equation}
and therefore, \(|A|_{3}^{2}\geq |G_{v}|_{3}=|\PSL_{2}(q)|_{3}m_{3}=3^{2n+1}m_{3}\). Thus \(|A|_{3}\geq 3^{\frac{2n+1}{2}}m_{3}^{\frac{1}{2}}\). However, \(m_{3}\leq (2n+1)_{3}\leq 3^{\frac{2n+1}{3}}\leq 3^{n}\), and \(|A|_{3}=|A_{1}|_{3}|\pi(A)|_{3}\leq |A_{1}|_{3}m_{3}\). Hence 
\begin{align*}
    |A_{1}|_{3}&\geq\frac{|A|_{3}}{m_{3}}
    \geq \frac{3^{\frac{2n+1}{2}}m_{3}^{\frac{1}{2}}}{m_{3}}
    = 3^{\frac{2n+1}{2}}m_{3}^{-\frac{1}{2}}\geq 3^{\frac{2n+1}{2}}3^{-\frac{n}{2}}
    =3^{\frac{n+1}{2}}>1.
\end{align*}
\\
Thus \(\{3,p\}\subseteq \Pi(A_{1})\). By \cite[Theorem 4 and Table 10.3]{transitive}, there are no proper subgroups of \(\PSL_{2}(q)\) with order divisible by both 3 and \(p\), and hence \(A_{1}=\PSL_{2}(q)=T\). On the other hand, since \(A^{g}=B\), we have \(T^{g}\leqslant A^{g}=B\). However, \(T\) is the unique subgroup in \(G_{v}\) isomorphic to \(\PSL_{2}(q)\), so this implies that \(T^{g}=T\), which is a contradiction to Lemma \ref{factor}.
\end{proof}

\begin{lemma}\label{r3}
 \(G_{v}\) is not a Type (iii) subgroup of \(G\).
\end{lemma}

\begin{proof}
 Suppose for a contradiction that \(G_{v}\) is a Type (iii) subgroup of \(G\), and again consider the homogeneous factorisation $G_v=AB$ in \eqref{e:homf} which implies that, for each $p$ dividing $|G_v|$, $G_v$ has more than one subgroup of order $p$.  We denote by \(S\) and \(T\) the normal subgroups of \(L_{v}=L\cap G_{v}\) isomorphic to \(2^{2}\) and \(\D_{\frac{q+1}{2}}\), respectively, so that \(L_{v}=(S\times T):3\). 
 
 By Lemma \ref{ppd} there exists a primitive prime divisor \(p\) of \(3^{2(2n+1)}-1=q^{2}-1\). Note that $p\ne 3$, and also \(p\)  divides \(q+1\),  and  \(p\) is  odd (as $p$ does not divide \(q-1\)). Hence \(p\in \Pi(\D_{\frac{q+1}{2}})\subseteq \Pi(G_{v})\). Since \(p>2(2n+1)\geq2m\) and  \(p\neq 3\), any subgroup \(Q\) of $G_v$ of order \(p\) must lie in \(T\). Since \(T=\D_{\frac{q+1}{2}}\) is dihedral, this implies that \(Q\) is the unique subgroup of order \(p\) in \(T\) and hence in \(G_{v}\). However, this contradicts Corollary \ref{Q} and therefore the result follows.
\end{proof}

\begin{lemma}\label{r4}
\(G_{v}\) is neither a Type (iv)  subgroup nor  a Type (v) subgroup of \(G\).
\end{lemma}
\begin{proof}
Suppose for a contradiction that \(G_{v}\) is a Type (iv) or (v) subgroup of \(G\). Recall, as discussed above, that for each prime $p$ dividing $|G_v|$, $G_v$ has more than one subgroup of order $p$.  
We denote by \(S\) and \( T\) the (unique) cyclic subgroups of \(L_{v}=L\cap G_{v}\) of order \(q\pm\sqrt{3q}+1\) and 6, respectively, so that \(L_{v}=S:T\). Since \(q\pm\sqrt{3q}+1\) is not divisible by 2 or 3, we see that \(|S|\) and \(|T|\) are coprime.

By Lemma \ref{proj} we have that \(\pi(A)=\pi(B)=\pi(G_{v})=\C_{m}\). Thus \(|A\cap L_{v}|=|B\cap L_{v}|\). Let \(p\) be a prime dividing \(|S|\). Then there is a unique subgroup \(Q_{p}\leqslant S\) of order \(p\). We note that since \(|S|\) and \(|T|\) are coprime, \(Q_{p}\) is the unique subgroup in \(L_{v}\) of order \(p\).
\begin{claim}\label{n}
\(|A\cap L_{v}|_{p}=1\).
\end{claim}
Suppose for a contradiction that \(|A\cap L_{v}|_{p}\geq p\). Then \(A\cap L_{p}\) has a subgroup of order \(p\). This subgroup must be \(Q_{p}\) as it is the unique subgroup of order \(p\) in \(L_{v}\). On the other hand, since \(|A\cap L_{v}|=|B\cap L_{v}|\), we find that \(Q_{p}\leqslant B\cap L_{v}\) as well. We note that \(A^{g}=B\), so \(Q_{p}^{g}\leqslant (A\cap L_{v})^{g}=B\cap L_{w}\leqslant B\cap L\leqslant G_{v}\cap L= L_{v}\). This implies that \(Q_{p}^{g}=Q_{p}\). However, this contradicts Lemma \ref{factor} and therefore Claim \ref{n} holds.
\\
\\By Claim \ref{n} we conclude that \(|A\cap L_{v}|\leq 6\). This implies that \(|A|=|A\cap L_{v}||\pi(A)|\leq 6m\). Suppose first that \(n=1\). Then \(m\leq3\), \(q=3^{3}\) and  \(|A|\leq 18\). Thus \(|G_{v}|\) is either divisible by \(37=q+\sqrt{3q}+1\) or by \(19=q-\sqrt{3q}+1\). However, \(|A|\) is divisible by neither  37 nor 19 since \(|A|\leq 18\). Hence \(G_{v}\) does not have a homogeneous factorisation when \(n=1\). Thus  \(n\geq 2\) and so
\begin{equation}
    q-\sqrt{3q}+1=3^{2n+1}-3^{n+1}+1\geq 9(2n+1)\nonumber.
\end{equation}
Since \(G_{v}=AB\), we have that \(|A|\cdot|B|=|G_{v}|\cdot|A\cap B|\). However,
\begin{align*}
|G_{v}|\cdot|A\cap B|&\geq(q\pm\sqrt{3q}+1)\cdot 6m\geq 9(2n+1)\cdot 6m>(6m)^{2}\geq |A|\cdot|B|.
\end{align*}
So we have a contradiction and the result follows.
\end{proof}

\begin{lemma}\label{r6}
 \(G_{v}\) is not a Type (vi) subgroup of \(G\).
 \end{lemma}
\begin{proof} Suppose for a contradiction that \(G_{v}\) is a Type (vi) subgroup of \(G\), so \( G_{v}=H.m\), where \(H=L_{v}=L\cap G_{v}=\Ree(q_{0})\), with $3^{2n+1}=q=q_0^r$ for some prime $r$, such that $r, m$ both divide $2n+1$. Now \(G_{v}\) has a homogeneous factorisation \(G_{v}=AB\) where \(A=G_{uv}\) and \(B=G_{vw}\) with \(B=A^{g}\) for some \(g\in G\). Let \(X:=A\cap L_{v}\) and \(Y:=B\cap L_{v}\). It follows from Lemma \ref{proj} that \(\pi(A)=\pi(B)=\C_{m}\). 
We divide the analysis into two cases:

\medskip\noindent
\emph{Case $1$: \(2n+1\) is not prime.}\quad In this case $q_0=3^{(2n+1)/r} > 3$. Let \(C\) be the centraliser of \(H=\Ree(q_{0})\) in \(G_{v}\). Then  \(\Aut(H)\geq G_{v}/C=(AB)/C=(AC/C)(BC/C)\geq HC/C\cong \Ree(q_{0})\). All the core-free factorisations of an almost simple group  with socle an exceptional group of Lie type are given in [16, Theorem B], and it follows that \(G_v/C\) does not have a core-free factorisation since \(q_{0}>3\). Hence \(\Ree(q_{0})\) is contained in one of \(AC/C\) or \(BC/C\). Without loss of generality, we may assume that \(\Ree(q_{0})\leq AC/C\). This together with the fact that \(H\cap C=1\) implies that \(H\leqslant A\). On the other hand, \(H^{g}\leqslant A^g=B\), and  since \(H\) is the unique subgroup of \(G_{v}\) isomorphic to \(\Ree(q_{0})\), we conclude that \(H^{g}=H\),  which contradicts Lemma \ref{factor}.
    
\medskip\noindent
\emph{Case $2$: \(2n+1\) is prime.}\quad  In this case,  \(m\in\{1,2n+1\}\) and $r=2n+1$, so \(q_{0}=3\) and \(H=L_{v}= H':3\) with \(H'\cong \PSL_{2}(8)\). If \(m=1\), then  \(AB\) is a homogeneous factorisation for \(H=\Ree(3)=\PSL_{2}(8):3\), but no such factorisation exists by Lemma \ref{ps}. Hence \(m=2n+1\) and we have $\Pi(A)=\Pi(G_v)=\Pi(H)\cup\{m\}=\{2,3,7,m\}$  (with possibly $m\in\{3,7\}$). 

Recall that $X=A\cap H=A\cap L_v$. Suppose  that \(\Pi(X\cap H')=\{2,3,7\}\).
By \cite{Atlas} there are no proper subgroups of \(\PSL_{2}(8)\) with order divisible by 2, 3 and 7, and hence  \(X\cap H'=H'\cong \PSL_{2}(8)\) so \(H'\leqslant A\). It follows that \((H')^{g}\leqslant A^{g}=B\), and since \(H'\) is the unique subgroup of \(G_{v}\) isomorphic to \(\PSL_{2}(8)\), we conclude that \((H')^{g}=H'\),   contradicting Lemma \ref{factor}. Thus  \(\Pi(X\cap H')\) is a proper subset of \(\{2,3,7\}\).

Further, since $|G_v:H'|=3m$ is odd, we have $|X\cap H'|_2=|A|_2\geq |G_v|_2^{1/2}=2^{3/2}$, and hence $|X\cap H'|_2\geq 2^2$,  so $2\in \Pi(X\cap H')$. If \(\Pi(X\cap H')=\{2\}\) then, since $|A|/|X|$ divides $3m$, it follows that $\Pi(A)\subseteq\{2,3,m\}$ and since $\Pi(A)=\{2,3,7,m\}$ we conclude that $m=7$ and $3= |A|_3 < 3^{3/2}=|G_v|_3^{1/2}$, which is a contradiction. Therefore \(\Pi(X\cap H')= \{2,p\}\)  for some \(p\in\{3,7\}\). 
If $p=3$ then $X\cap H'$ is a subgroup of $H'=L_2(8)$ of order divisible by  12  and dividing $72$. However there are no such subgroups, see for example \cite[p. 6]{Atlas}.
Therefore  \(\Pi(X\cap H')= \{2,7\}\), and $X\cap H'$ is a subgroup of $H'=L_2(8)$ of order divisible by 28.  It follows from \cite[p. 6]{Atlas} that $X\cap H'= [2^3]:7$ since this group has no subgroups of index 2.   The same   argument gives $Y\cap H'\cong [2^3]:7$. 

If the prime \(m\ne 3\), then $|X|_3=|A|_3\geq |G_v|_3^{1/2} = 3^{3/2}$ so that 
\(|X\cap H'|_{3}\geq|X|_{3}/3\geq 3\), which is a contradiction. 
Hence $m=3$.
In this case, 
\(G_{v}=L_{v}\times \langle z\rangle\cong (\PSL_{2}(8):3)\times 3\), where \(z\) is a field automorphism of order 3. Since \(AB=G_{v}=L_{v}\times\langle z\rangle\), we have that \(\PSL_{2}(8):3=L_{v}\cong G_{v}/\langle z\rangle=((A\langle z\rangle/\langle z\rangle)(B\langle z\rangle/\langle z\rangle)\). Now $A\langle z\rangle/\langle z\rangle$ has a normal subgroup $(X\cap H')\langle z\rangle/\langle z\rangle\cong [2^3]:7$, and similarly
$B\langle z\rangle/\langle z\rangle$ has a normal subgroup $[2^3]:7$. However, by \cite[Theorem A]{j}, there are no such factorisations of \(\PSL_{2}(8):3\). This completes the proof.
 \end{proof}
 
\begin{theorem}\label{ree}
Suppose that \(\Gamma\) is a \(G\)-vertex-primitive \((G,s)\)-arc-transitive digraph such that \(\Soc(G)=\Ree(q)\) with \(q=3^{2n+1}\), for some \(n\geq 1\). Then \(s\leq 1\).
\end{theorem}

\begin{proof}
Suppose for a contradiction that \(s\geq 2\). Then the conditions of Hypothesis \ref{ga} hold with \(\Soc(G)=\Ree(q)\). Since \(G\) acts vertex-primitively on \(\Gamma\), the vertex stabiliser \(G_{v}\) is a maximal subgroup of \(G\) and so is given by Corollary \ref{reemax}. By
Lemmas \ref{r2}, \ref{r3}, \ref{r4} and \ref{r6}, $G_v$ cannot be of types (ii)--(vi). Hence $G_v$ is of type $(i)$. However, in this case,  $G$ acts 2-transitively on the set of right cosets of $G_v$ in $G$, which implies that $\Gamma$ is an undirected graph, contradicting it being a digraph.  Hence the result follows.
\end{proof}

\section{Suzuki Groups}\label{s:Sz}
Again, for \(\Sz(q)\leqslant G\leqslant \Aut(\Sz(q))\), since \(\Aut(\Sz(q))=\Sz(q):(2n+1)\) we have \(G=\Sz(q):m\) for some divisor \(m\) of \(2n+1\). Since the action of \(G\) on \(\Gamma\) is vertex-primitive, a vertex stabiliser \(G_{v}\) is maximal in \(G\). The following list of the maximal subgroups of an almost simple of socle \(\Sz(q)\) may be found in \cite[Table 8.16]{holt}.



\begin{theorem}\label{cor:sz}
For \(G=\Sz(q):m\), where \(m\) divides \(2n+1\), the maximal subgroups of \(G\) not containing \(\Sz(q)\) are (up to conjugacy):
\\(i) \(([q^{2}]:(q-1)).m\),
\\(ii) \(\D_{2(q-1)}.m\),
\\(iii) \(((q+\sqrt{2q}+1):4).m\),
\\(iv) \(((q-\sqrt{2q}+1):4).m\),
\\(v) \(\Sz(q_{0}).m\), where \(q=q_{0}^{r}\), \(r\) is prime, and \(q_{0}>2\).
\end{theorem}

For the rest of this section we assume that Hypothesis \ref{ga} holds for \(G= \Sz(q):m\), where $q=2^{2n+1}$ and $m$ divides $2n+1$, and we let \(L=\Soc(G)=\Sz(q)\). Let $\pi:G_v\rightarrow G_v/L_v$ be the natural projection map. Note that since $G=LG_v$ we have that $\pi(G_v)\cong G/L\cong C_m$.

We consider separately each of the possibilities for the maximal subgroup $G_v$ according to Corollary~\ref{cor:sz}. We note in particular that, by Hypothesis \ref{ga}, \(G_{v}\) has a homogeneous factorisation \(G_{v}=AB\) where \(A=G_{uv}\) and \(B=G_{vw}\) with \(A^{g}=B\) for some \(g\in G\). This implies, by  Corollary \ref{Q}, that for each prime $p$ dividing $|G_v|$, $G_v$ has at least two subgroups of order $p$. We use these facts several times in our arguments.

\begin{lemma}\label{s2}
Suppose that Hypothesis \(\ref{ga}\) holds with \(\Soc(G)=\Sz(q)\). Then \(G_{v}\) is not a Type (ii) subgroup of \(G\). 
\end{lemma}
\begin{proof}
Suppose to the contrary that \(G_{v}\) is a Type (ii) subgroup of \(G\).  Then $L_v\cong D_{2(q-1)}$. By Lemma \ref{ppd} there exists a primitive prime divisor \(p\) of \(2^{2n+1}-1\), and as noted before Lemma \ref{ppd}, $p$ satisfies \(p>2n+1\geq m\). Hence \(G_{v}\) has a unique subgroup \(Q_{p}\) of order \(p\), which is a contradiction, as noted above. 
\end{proof}

\begin{lemma}\label{s3}
Suppose that Hypothesis \(\ref{ga}\) holds with \(\Soc(G)=\Sz(q)\). Then \(G_{v}\) is neither a Type (iii) subgroup nor a Type (iv) subgroup of \(G\).
\end{lemma}

\begin{proof}
Suppose to the contrary that \(G_{v}\) is a Type (iii) or Type (iv) subgroup of \(G\). As above we have \(G_{v}=AB\) with \(A^{g}=B\) for some \(g\in G\). 
It follows from Lemma \ref{proj} that \(\pi(A)=\pi(B)=\pi(G_{v})\cong \C_{m}\), and hence \(|A\cap L_{v}|=|B\cap L_{v}|\). 

Let \(S\) and \(T\) denote cyclic subgroups of \(L_{v}=L\cap G_{v}\) of orders \(q\pm\sqrt{2q}+1\) and 4, respectively, such that \(L_{v}=S:T\).
Since \(q\pm\sqrt{2q}+1\) is an odd integer, the orders \(|S|\) and \(|T|\) are coprime. 

Let \(p\) be a prime dividing \(|S|\), and note that the cyclic group $S$ has a unique subgroup \(Q_{p}\) of order \(p\), and that $Q_p$ is the unique subgroup of order $p$ in \(L_v\).
If $p$ divides \(|A\cap L_{v}|\), then \(A\cap L_{v}\) contains  $Q_p$, and since \(|A\cap L_{v}|=|B\cap L_{v}|\), also \(Q_{p}\leqslant B\cap L_{v}\). Moreover, since \(A^{g}=B\), it follows that \(Q_{p}^{g}\) is also a subgroup of $B\cap L_v$ of order $p$, and so  \(Q_{p}^{g}=Q_{p}\). However, this contradicts Lemma \ref{factor}, and therefore  $p$ does not divide \(|A\cap L_{v}|\). Since this holds for all primes $p$ dividing $|S|$, we conclude that \(|A\cap L_{v}|\) divides $|T|=4$. 

Since $A/(A\cap L_v)\cong AL_v/L_v\leq G_v/L_v\cong\C_m$, it follows that $|A|$ divides $4m$, and hence $G_v=AB$ has order dividing $|A|\cdot|B|=|A|^2$, which divides $16m^2$. On the other hand $|G_v|= 4m(q\pm \sqrt{2q}+1)$, and hence the odd integer $q\pm\sqrt{2q}+1$ divides $m$. This is impossible since $q\pm\sqrt{2q}+1\geq 2^{2n+1}-2^{n+1}+1>2n+1\geq m$, for all $n\geq 1$. This contradiction completes the proof. 
\end{proof}

\begin{lemma}\label{s4}
 Suppose that Hypothesis \(\ref{ga}\) holds with \(\Soc(G)=\Sz(q)\). Then \(G_{v}\) is not a Type (v) subgroup of \(G\). 
\end{lemma}
\begin{proof}
Suppose to the contrary that \(G_{v}=\Sz(q_{0}).m\),  where \(q=q_{0}^{r}\), for some prime \(r\) dividing $2n+1$. As above we have \(G_{v}=AB\) with \(A^{g}=B\) for some \(g\in G\). Let \(C\) denote the centraliser of \(\Sz(q_{0})\) in \(G_{v}\).
Then  
\[\Aut(\Sz(q_{0}))\gtrsim G_{v}/C=(AB)/C=(AC/C)(BC/C)\gtrsim \Sz(q_{0}).
\] 
However, \(G_{v}/C\) does not have a core-free factorisation by \cite[Theorem B]{transitive}, and therefore one of the factors, say $AC/C$, contains  \(\Sz(q_{0})\).
This,  together with the fact that \(L_{v}\cap C=1\), implies that  \(\Sz(q_{0})=L_{v}\leqslant A\). 
Moreover, since \(A^{g}=B\), we have  \(L_{v}^{g}\leqslant A^{g}=B\). However \(L_{v}\) is the only subgroup of  \(G_{v}\) isomorphic to \(\Sz(q_{0})\), and hence \(L_{v}^{g}=L_{v}\). This contradicts Lemma \ref{factor}, and completes the proof.
\end{proof}

Now we can collect all these results to prove the following result. 

\begin{theorem}\label{suzuki}
Suppose that \(\Gamma\) is a \(G\)-vertex-primitive \((G,s)\)-arc-transitive digraph such that \(\Soc(G)=\Sz(q)\) with \(q=2^{2n+1}\) for some positive integer \(n\). Then \(s\leq 1\).
\end{theorem}
\begin{proof}
Suppose for a contradiction that \(s\geq 2\). Then the conditions of Hypothesis \ref{ga} hold with \(\Soc(G)=\Sz(q)\). Since \(G\) acts vertex-primitively on \(\Gamma\), the vertex stabiliser \(G_{v}\) is a maximal subgroup of \(G\) and so is given by Corollary \ref{cor:sz}. By Lemmas \ref{s2}, \ref{s3} and \ref{s4}, $G_v$ is not of type $(ii)--(v)$ and so $G_v$ must be of type (i). However, in this case $G$ acts 2-transitively on the set of right cosets of $G_v$ and so $\Gamma$ is an undirected graph, contradicting it being a digraph. Hence the result follows.
\end{proof}

Theorem \ref{general}  follows immediately from Theorems \ref{ree} and \ref{suzuki}.

\section{Examples}
In this final section, we construct examples of vertex-primitive \((G,1)\)-arc-transit\-ive digraphs with \(\Soc(G)=\Sz(2^{2n+1})\) and \(\Ree(3^{2n+1})\), respectively. The following is the mechanism by which we construct the examples:

Let \(G\) be a group, \(H\) a subgroup of \(G\) which does not contain any non-trivial normal subgroup of \(G\), \(V :=\{Hz:z\in G\)\}, and let \(g\in G\)   such that \(g^{-1}\notin HgH\). We define a binary relation $\to\,$  on \(V\) by  
\begin{center}
\(Hx\to Hy\)\quad if and only if\quad  \(yx^{-1}\in HgH\) for any \(x,y\in G\).    
\end{center}
 Then \((V,\to)\) is a digraph, which we denote  by \(\Cos(G,H,g)\). Since \(yg(xg)^{-1}=ygg^{-1}x=yx^{-1}\), right multiplication by elements of $G$ preserves the relation $\to\,$ and hence induces automorphisms of \((V,\to)\), yielding a subgroup \(\mathrm{R}_{H}(G)\cong G\) of \(\Aut(\Cos(G,H,g))\). Further the subgroup $\mathrm{R}_H(H)$ of right multiplications by elements of $H$ is the stabiliser in \(\mathrm{R}_{H}(G)\) of the vertex $H$ of  \((V,\to)\), and it follows from the definition of the relation $\to\,$ that \(\mathrm{R}_{H}(H)\) acts transitively on the set of arcs  $(H, Hx)$ beginning with $H$, since these arcs are precisely those of the form $(H, Hgh)$ for $h\in H$. Thus \(\mathrm{R}_{H}(G)\) acts arc-transitively on \(\Cos(G,H,g)\).

 We aim to find a maximal subgroup \(H\leqslant G\) and an element \(g\in G\) such that \(g^{-1}\notin HgH\) to obtain a 1-arc-transitive digraph \(\Cos(G,H,g)\), for \(G=\Sz(2^{2n+1})\) and \(G=\Ree(3^{2n+1})\).
 
\subsection{A one-arc transitive digraph admitting a Suzuki group}
Here \(G=\Sz(q)\), where \(q=2^{2n+1}\) with $n\geq 1$. Let $\Omega$ denote a set of size $q^2+1$ on which $G$ acts $2$-transitively. Let $a, b\in \Omega$ with $a\ne b$. We define the following notation:
\begin{itemize}
\item[(i)]  \(L:=G_a\), so \(L\cong [q^{2}]:(q-1)\);
    \item[(ii)]  \(K:=G_a\cap G_b\), so \(K=\langle\kappa\rangle\cong q-1\), for some \(\kappa\in G\);
    \item[(iii)] \(Q=[q^2]\), the normal Sylow $2$-subgroup of \(L\), so \(L=Q: K\);
    \item[(iv)] an involution  \(\tau\in \N_{G}(K)\), so \(\N_{G}(K)=\langle \kappa, \tau \,|\, \kappa^{q-1}=\tau^{2}=1,\tau\kappa\tau=\kappa^{-1}\rangle\), see \cite[Proposition 3]{szk};
    \item[(v)] an element \(\rho\in Q\) of order 4.
\end{itemize}
We use this notation throughout this subsection, and also the following results.

\begin{lemma}{\rm \cite[Proposition 1]{szk}\label{cent}}
For any non-trivial element  \(x\in Q\), the centraliser \(\C_{G}(x)\leq Q\).
\end{lemma}

\begin{lemma}{\rm \cite[p 108-109]{szk}\label{tau}}
The element $\tau$ satisfies \(b=a^{\tau}\) and \(a=b^{\tau}\), so \(\tau L\tau = G_b\), 
\(L\cap\tau L\tau=K\), and \(Q\cap\tau L\tau=\{1\}\).
\end{lemma}

We now give our construction.

\begin{lemma}
Let \(H := \N_{G}(K)\) and \(g := \rho\). Then \(g^{-1}\notin HgH\), and \(\Cos(G, H, g)\) is a \((G,1)\)-arc-transitive digraph. 
\end{lemma}

\begin{proof}
As we explained above, if \(g^{-1}\notin HgH\), then \(\Cos(G, H, g)\) is a \((G,1)\)-arc-transitive digraph. So it is sufficient to prove that \(g^{-1}\notin HgH\).
Suppose that this is not the case, that is, there exist \(x, y\in H=\N_G(K)\) such that \(\rho^{-1}=x\rho y\). 

Note that $L=Q: K$ and $L\cap \N_G(K)=K$, and also that $\rho\in Q<L$. Thus if $x\in K$ then $y=\rho^{-1}x^{-1}\rho\in L$ and hence $y\in L\cap \N_G(K)=K$. Similarly if $y\in K$ then also $x\in K$. Thus $x, y$ are either both in $K$, or both in $\N_G(K)\setminus K$.

Suppose first that $x, y\in K$. Now \(\rho\in Q\), and since  $x\in K<L$ and \(Q\) is a normal subgroup of \(L\), it follows that  $x\rho x^{-1}\in Q$, and also  \((x\rho x^{-1})(xy)=x\rho y=\rho^{-1}\in Q\). This implies that $xy\in Q$ and hence $xy\in K\cap Q=\{1\}$. Thus $y=x^{-1}$, and so \(\rho^{-1}=x\rho x^{-1}\), which implies that $x^2\in \C_G(\rho)$. By Lemma \ref{cent}, \(\C_{G}(\rho)\leqslant Q\), so \(x^{2}\in K\cap Q=\{1\}\). However, \(x\in K\) and $|K|=q-1$ is odd. Hence \(x=1\), so  $\rho^{-1}=x\rho x^{-1}=\rho$, which contradicts the fact that $\rho$ has order $4$. 
Thus we must have $x, y\in \N_G(K)\setminus K$, and hence $x=\kappa^i\tau$ and $y=\kappa^j\tau$, for some $i, j$. This implies that $\rho^{-1}=x\rho y = \tau(\kappa^{-i}\rho\kappa^j)\tau\in \tau L\tau$, and we also have $\rho^{-1}\in Q$. Thus $\rho^{-1}\in Q\cap \tau L \tau$ and so, by Lemma \ref{tau}, $\rho^{-1}=1$, which is a contradiction. This completes the proof.
\end{proof}

\subsection{A one-arc transitive digraph admitting a Ree group}
Here \(G=\Ree(q)\), where \(q=3^{2n+1}\) with $n\geq 1$.
Although several  \((G,2)\)-arc-transitive undirected graphs have been constructed, see \cite{fang}, we are interested in constructing a  \((G,1)\)-arc-transitive digraph with \(G\) acting primitively on the vertex set.  
Our treatment follows Wilson's description of the group $G$ given in his book \cite[Section 4.5]{Wilson}.
It is different from some other constructions for these groups, say in \cite{Nu}, which require knowledge of Lie algebras and algebraic groups. 
Wilson has an elementary approach developed in \cite{appr} and \cite{cons}, and we use the detailed description given in \cite[p 134-138]{Wilson}.

Wilson~\cite{Wilson} starts with a faithful $7$-dimensional representation of the group \(G=\Ree(3^{2n+1})\) on a space \(V\) over a field \(\mathbf{F}_{q}\) of order $q$. The space $V$ admits a $G$-invariant non-degenerate symmetric bilinear form \(f\)  with an orthonormal basis \(\mathcal{B}:=\{u_{0}, u_1,\ldots,u_6\}\). He defines a second basis 
\(\mathcal{C}:=\{v_{1}, v_2,\ldots,v_7\}\) for $V$ by 
\[\begin{array}{llp{1cm}ll}
v_{1}&=u_{3}+u_{5}+u_{6},& &
v_{2}&=u_{1}+u_{2}+u_{4},\\
v_{3}&=-u_{0}-u_{3}+u_{6},& &
v_{4}&=u_{2}-u_{1},\\
v_{5}&=-u_{0}+u_{3}-u_{6},& &
v_{6}&=-u_{1}-u_{2}+u_{4},\\
v_{7}&=-u_{3}+u_{5}-u_{6}.
\end{array}\]
He shows that the maps \(\gamma\) and \(\sigma\) given by:
\begin{equation}
    u_{i}^{\gamma}=\begin{cases}u_{i}, &i=0,1,3\\
                                -u_{i}, &i=2,4,5,6\nonumber
    
    \end{cases}
\end{equation} and 
\begin{equation}
    u_{i}^{\sigma}=\begin{cases} u_{i}, &i=0,4,5\\
                                -u_{i}, &i=1,2,3,6\nonumber
    \end{cases}
\end{equation} 
are commuting involutions lying in the group $G$. Thus $G$ contains the subgroup 
\(K=\langle \gamma, \sigma\rangle=\langle \gamma\rangle\times\langle \sigma\rangle\cong 2^{2}\). Moreover, we let \(\delta:=\sigma\gamma\), and find that 
\begin{equation}
    u_{i}^{\delta}=\begin{cases}u_{i}, &i=0,2,6\\
                               -u_{i}, &i=1,3,4,5\nonumber
    \end{cases}
\end{equation}Let \(T:= \N_G(K)\) and \(H=\C_{G}(\sigma)\). Note that by \cite[Lemma 2.2]{fang}, \(T\) and \(H\) are maximal subgroups of \(G\) and \(T\cong (2^{2}\times \D_{\frac{q+1}{2}}):3\) and \(H\cong 2\times \PSL_{2}(q)\).  Let us denote by \(W\), \(W_{1}\), \(W_{2}\) and \(W_{3}\) the subspace \(\langle u_{1},\ldots, u_{6}\rangle\), \(\langle u_{1}, u_{3}\rangle\), \(\langle u_{2}, u_{6}\rangle\) and \(\langle u_{4}, u_{5}\rangle\), respectively. 
 For a subspace $U$ of $V$, we denote the setwise stabiliser in $G$ of $U$ by \(\Stab_{G}(U)\).  We need the following properties of $T$.

\begin{lemma}\label{l:H}
The subgroup 
\(T=\Stab_{G}(\langle u_{0}\rangle)\), and \(T\) leaves invariant the subspace \(W :=\langle u_{1},u_2,\ldots, u_6\rangle\).
\end{lemma}

\begin{proof}
Let \(\Fix(K)\) be the subspace of fixed points of \(K\) in \(V\). Then \(\Fix(K)=\langle u_{0}\rangle\), by the definitions of $\gamma$ and $\sigma$, and since $\mathcal{B}$ is an orthonormal basis it follows that \(\Fix(K)^{\perp}=W\). Since \(T=\N_G(K)\) normalises \(K\), it follows that \(T\) leaves invariant both \(\Fix(K)\) and  \(\Fix(K)^{\perp}\). 
Finally since $T$ is maximal in $G$, it follows that $T$ is equal to the full stabiliser of 
\(\Fix(K)\).
\end{proof}

For any \(x\in\{\delta,\gamma,\sigma\}\), let \(\C_{W}(x)=\{u\in W|u^{x}=u\}\).

\begin{lemma}\label{permuteW_i}
With above notation, \(T\) permutes \(W_{1}\), \(W_{2}\) and \(W_{3}\). In particular, the action of \(T\) on \(\{W_{1}, W_{2}, W_{3}\}\) is isomorphic to \(\C_{3}\)
\end{lemma}
\begin{proof}
Since \(T\) normalises \(K=\langle \sigma, \gamma\rangle\), we see that \(T\) permutes \(\sigma, \gamma\) and \(\delta\). We also note that \(\C_{W}(\gamma)=W_{1}\), \(\C_{W}(\delta)=W_{2}\) and \(\C_{W}(\sigma)=W_{3}\). For \(i\in\{1,2,3\}\), \(W_{i}=\C_{W}(x)\) for some \(x\in\{\sigma,\gamma,\delta\}\). Since $T$ leaves $W$ invariant we have that \(\C_{W}(x)^{t}=\C_{W}(x^{t})\) for each $x$.
Thus  \(W_{i}^{t}=\C_{W}(x)^{t}=\C_{W}(x^{t})\). Since \(x^{t}\in\{\sigma,\gamma,\delta\}\), we find that \(W_{i}^{t}\in \{W_{1}, W_{2}, W_{3}\}\). Moreover, we find that \(\C_{G}(K)\cong 2^{2}\times \D_{\frac{q+1}{2}}\), so \([T:\C_{G}(K)]=3\). We note that \(\C_{G}(K)\) acts trivially on \(\{W_{1}, W_{2}, W_{3}\}\). This implies the kernel of the action of \(T\) on \(\{W_{1}, W_{2}, W_{3}\}\) is of index 3, so the action is isomorphic to \(\C_{3}\). Hence the result follows.
\end{proof}
In Wilson's description~\cite[p 136]{Wilson}, $G$ has a Borel subgroup $B$ such that there exists \(g\in B\)  determined by its action on the basis $\mathcal{C}$ as follows:
\[\begin{array}{llp{1cm}ll}
v_{1}&\mapsto v_{1}& &
v_{2}&\mapsto v_{2}\\
v_{3}&\mapsto v_{1}+v_{3}& &
v_{4}&\mapsto v_{2}+v_{4}\\
v_{5}&\mapsto 2v_{1}+v_{5}& &
v_{6}&\mapsto v_{2}+2v_{4}+v_{6}\\
v_{7}&\mapsto v_{1}+v_{3}+2v_{5}+v_{7}.
\end{array}\]
It is easily checked that  \(g\) acts on \(u_{i}\) by:
\[\begin{array}{llp{1cm}ll}
u_{0}&\mapsto u_{0}& &
u_{1}&\mapsto u_{2}\\
u_{2}&\mapsto u_{4}& &
u_{3}&\mapsto u_{6}\\
u_{4}&\mapsto u_{1}& &
u_{5}&\mapsto u_{3}\\
u_{6}&\mapsto u_{5}.
\end{array}\]
In particular we find that \(g\in T=\N_{G}(K)\) and \(W_{i}^{g}=W_{i+1}\) for all \(W_i\in\{W_{1}, W_{2}, W_{3}\}\).

\begin{lemma}
Let \(H\) and \(g\) be as above. Then \(g^{-1}\notin HgH\), and \(\Cos(G, H,g)\) is a \((G,1)\)-arc-transitive digraph.
\end{lemma}

\begin{proof}
Let $g'=g^{-1}$. If \(g'\notin HgH\), then \(\Cos(G, H, g)\) is a \((G,1)\)-arc-transitive digraph. So it is sufficient to prove that \(g'\notin HgH\).
Suppose that this is not the case. that is, there exist \(x, y\in H=\C_G(\sigma)\) such that \(g'=x^{-1}gy\), or equivalently, \(xg=g'y\). Let us denote by \(E_{1}\) and \(E_{-1}\) the eigenspaces of \(\sigma\) with eigenvalues 1 and $-1$ respectively. Indeed, \(E_{1}=\langle u_{0}, u_{4}, u_{5}\rangle\) and \(E_{-1}=\langle u_{1},u_{2}, u_{3},u_{6}\rangle\). \begin{claim}\label{xy}
We have \(u_{0}^{x}, u_{0}^{y}\in \langle u_{0}\rangle\).
\end{claim}
First we note that since \(x, y\in H\), \(x\sigma=\sigma x\) and \(y\sigma=\sigma y\). This implies that \(u_{0}^{x\sigma}=u_{0}^{\sigma x}=u_{0}^{x}\). Thus \(u_{0}^{x}\in E_{1}\), so that \(u_{0}^{x}=\alpha_{0}u_{0}+\alpha_{4}u_{4}+\alpha_{5}u_{5}\) for some \(\alpha_{i}\in \mathbb{F}_{q}\). Similarly \(u_{0}^{y}=\beta_{0}u_{0}+\beta_{4}u_{4}+\beta_{5}u_{5}\) for some \(\beta_{i}\in \mathbb{F}_{q}\).
\\Now, we let \(I=\{0,4,5\}\) and since \(xg=g'y\), we have:
\begin{equation}
    \sum_{i\in I}\beta_{i}u_{i}=u_{0}^{y}=u_{0}^{g'y}=u_{0}^{xg}=\sum_{i\in I}\alpha_{i}u_{i}^{g}=\beta_{0}u_{0}+\beta_{4}u_{1}+\beta_{5}u_{3}.
    \end{equation} This implies that \(\alpha_{0}=\beta_{0}\) and \(\alpha_{4}=\alpha_{5}=\beta_{4}=\beta_{5}=0\). Thus the claim is proved.

\medskip

 Let us consider the action of \(x\) and \(y\) on \(W_{3}\). For any \(v\in W_{3}=\langle u_{4}, u_{5}\rangle\), we have that \(v^{x\sigma}=v^{\sigma x}=v^{x}\) and \(v^{y\sigma }=v^{\sigma y}=v^{y}\). Thus \(v^{x}, v^{y}\in E_{1}=\langle u_{0}, u_{4}, u_{5}\rangle\). On the other hand, by Claim \ref{xy} we find that \(x,y\in T\). By Lemma \ref{permuteW_i} \(W_{3}^{x}, W_{3}^{y}\in \{W_{1}, W_{2}, W_{3}\}\). Hence \(v^{x}, v^{y}\in E_{1}\cap W_{i}\) for some \(i\in\{1,2,3\}\). Since the only possible \(i\) such that \(E_{1}\cap W_{i}\neq\{0\}\) is 3, we deduce that \(W_{3}^{x}=W_{3}^{y}=W_{3}\). Hence \(x\) and \(y\) either swap or fix \(W_{1}\) and \(W_{2}\). If \(x\) and \(y\) swap them, then \(x, y\) acts as \((1,2)\) on \(\{W_{1}, W_{2}, W_{3}\}\). However, by Lemma \ref{permuteW_i}, the action of \(T\) on \(\{W_{1}, W_{2}, W_{3}\}\) is isomorphic to \(\C_{3}\) and does not have any element of order 2. Thus we deduce that \(W_{i}^{x}=W_{i}^{y}=W_{i}\) for \(i=1,2\).

 \medskip
 
 Now let us consider the action of \(xg\) and \(g'y\) on \(W_{1}\), 
\begin{equation}
    W_{1}^{xg}=W_{1}^{g}=W_{2}
\end{equation}
while
\begin{equation}
    W_{1}^{g'y}=W_{3}^{y}=W_{3}.
\end{equation}
This is a contradiction to \(xg=g'y\). Hence such \(x\) and \(y\) do not exist and the result follows.
\end{proof}


\end{document}